\documentclass[12pt,a4paper]{amsart}
\usepackage[utf8]{inputenc}
\usepackage[english]{babel}
\usepackage{amsmath}
\usepackage{amsfonts}
\usepackage{amssymb}
\usepackage{pgf,tikz}
\usetikzlibrary{arrows}

\newtheorem{proposition}{Proposition}
\newtheorem{lemma}{Lemma}
\newtheorem{theorem}{Theorem}

\newcounter{liczprz}
\newenvironment{example}{\refstepcounter{liczprz}{\textbf{Example~\theliczprz}}}{\(\lozenge\)\newline}

\begin{document}
\author{Krzysztof Le\'{s}niak}
\title{Random iteration and projection method}
 
\maketitle 

\section{Introduction}

Let \(X\) be a complete metric space with metric denoted by \(d\).
Albeit such abstract scenery, 
best motivated examples come from and lie in Euclidean spaces 
with a possible prospect for Hilbert and Banach spaces.

Recall from the fractal geometry that a system 
$\Phi = (X;f_i,i=1,{\ldots},N)$ of maps $f_i:X{\to} X$ is called an 
\textbf{iterated function system}, shortly IFS (\cite{Edgar, FractalsEver}). 
We assume that the maps $f_i$ are nonexpansive:
\[
\forall_{x_{1},x_{2}{\in} X}\; d(f_i(x_1),f_i(x_2)) \leq d(x_1,x_2).
\]
This allows for situations where neither a strict attractor 
(\cite{BarnsleyVinceProjective}) nor a Lasota-Myjak semiattractor 
(\cite{LasotaMyjak}) exists, despite the fact that such attractors 
may be present in permanently 
noncontractive cases (cf. \cite{BarnsleyVinceChaos}). However the nature of 
basic examples which took our attention and sparkled this research justifies 
the assumptions we make here and add further. 
To avoid a mystery: we are interested in invariant sets rather than attractors
of IFSs. 
Concerning the existence of invariant sets one can assure it under 
very mild dissipativity conditions even in the absence of continuity of actions,
see e.g. the references in \cite{LesniakCEJM} or \cite{Kieninger}. 
So let us refine our goal. 

The study of systems of contractions (and relatives) became a standard topic 
of several books, e.g., \cite{FractalsEver, Edgar, MauldinUrbanskiGraph, 
Massopust, LaTorre} to mention only a small portion of literature.
It is known also quite a lot about the dynamics of a single nonexpansive map 
and omega-limit sets, see \cite{DafermosSlemrod, RoehrigSine, 
AkcogluKrengel, GoebelKirk, Nussbaum, Sine, DiLena}. 
However it is still  visibly less known about iterated systems 
of nonexpansive maps.
Frankly the idea of a nonexpansive IFS appears in different disguise shortly 
before 1939 in the context of linear algebra and functional analysis. This is 
the method of projections (introduced by von Neumann and Kaczmarz).
This topic also received a great attention 
(\cite{BauschkeThesis, BauschkeBorwein, Galantai, 
AlternatingProjection, BauschkeReich, BauschkeDeutsch}), 
again, as with the fractal geometry, due to its applicability potential. 
Yet for the full understanding and synthesis of the projection method some 
gaps await its explanation. One, towards which we aspire, reads as follows: 
\emph{what happens when we use projections onto sets with 
empty intersection?}
In the case of two sets one easily finds out that the iteration recovers 
the two nearest points in the sets realizing the (infimum) distance between 
these sets (\cite{BauschkeThesis} 11.4.3). Our goal is to show, under very weak
contractivity conditions (allowing for orthogonal projections among others),
a general principle that the 
\emph{iterated maps recover a minimal invariant set} in the sense that 
the omega-limit set of the generated orbit constitutes an invariant set.
(Note that for $N=1$, a single map, this is an elementary exercise).
Related results together with a panorama of examples can be found in
\cite{Angelos}.

Given an IFS $(X;f_i,i=1,{\ldots},N)$ we define the 
\textbf{Hutchinson operator} 
$\Phi: 2^{X}\setminus\{\emptyset\} {\to} 2^{X}\setminus\{\emptyset\}$ 
via
\[
\forall_{S{\in}2^{X}\setminus\{\emptyset\}}\; 
\Phi(S) := \bigcup_{i=1}^{N} f_{i}(S),
\]
where $2^{X}\setminus\{\emptyset\}$ stands for the family of nonempty 
subsets of $X$. (Note that we do not take the closure of the union in 
this variant of the definition).

A nonempty $S\subset X$ is said to be an \textbf{invariant set} for 
the system of maps $\{f_{1},{\ldots},f_{N}\}$ when $\Phi(S)=S$,
and a \textbf{subinvariant set} when $\Phi(S){\subset} S$. 
Traditional dynamics set focus on compact invariant sets 
(see however \cite{MauldinUrbanskiGraph}
for a reasonable deviation from this rule). This is the case here too,
but keeping a general definition will prove handy in the next Section.

By an \textbf{orbit} $(x_{n})_{n=0}^{\infty}$  starting at $x_{0}\in X$ 
with a driving sequence of symbols 
$(i_{n})_{n=1}^{\infty} \in\{1,\ldots,N\}^{\infty}$ 
we understand
\[
x_{n} := f_{i_n}{\circ}{\ldots}{\circ} f_{i_1}(x_0).
\]
Such iterations are fundamental for some numerical methods in
fractal geometry (chaos game algorithm \cite{FractalsEver, LaTorre})
and convex geometry (cyclic projection algorithm 
\cite{BauschkeBorwein, AlternatingProjection}).  
The driving sequence may be called a driver (\cite{McFarlaneHoggar}) 
or a control sequence (\cite{BauschkeThesis}); by a slight abuse, 
against the direction of compositions of maps, a code or an address 
(\cite{FractalsEver, Kieninger}) might be accepted too.
If the process generating symbols is stochastic, then the terms `random 
driver' and `random orbit' are justified. However random driver can mean 
a sequence where each symbol repeats infinitely often 
(\cite{BauschkeThesis}; repetitive below)
and `random orbit' can stand for an orbit driven by sufficiently complex 
deterministic sequence of symbols; with this respect also `chaotic orbit'
is in use (\cite{FractalsEver, BarnsleyLesniak}).  

We list below various types of drivers 
(\cite{BauschkeThesis, BarnsleyLesniak, CaludeStaiger}). A sequence 
$(i_{n})_{n=1}^{\infty} \in\{1,\ldots,N\}^{\infty}$
is called
\begin{itemize}
\item \textbf{cyclic}, if $i_{n}=\pi((n-1) \mod N + 1)$, for
all $n\geq 1$ under some fixed 
permutation $\pi$ of $\{1,{\ldots},N\}$,
\item \textbf{repetitive}, if for every $\sigma\in\{1,{\ldots},N\}$
the set $\{n\geq 1: i_{n}=\sigma\}$ is infinite,
\item \textbf{disjunctive}, if it contains avery possible finite word as 
its subword, namely for all $m\geq 1$ and every word 
$(\sigma_{1},{\ldots},\sigma_{m})\in\{1,{\ldots},N\}^{m}$
there exists $n_{0}\geq 1$ s.t. $i_{n_{0}-1\, +j}= \sigma_{j}$ for 
$j= 1,{\ldots},m$.
\end{itemize}
A disjunctive driver and a cyclic driver are necessarily repetitive 
but the reverse implications are obviously false. 
Also neither a cyclic sequence is disjunctive nor vice-versa.

Further we shall employ a disjunctive driver in the main theorem.
To understand why this feature fits well a standard cyclic projection 
onto two sets, one should recognize that the 
(linear or metric nearest point) projection map $P$ is idempotent, 
$P{\circ} P= P$ (a retraction put in a nonlinear topology framework).
The cancellations in any orbit build from $N=2$ projections show 
that the result is similar (modulo repetitions) to an 
alternating projection orbit regardless of how complex disjunctive 
driver was applied, see Examples \ref{ex:KaczmarzLines} and
\ref{ex:ParaLines}. 

We define the \textbf{omega-limit set} of $(x_{n})_{n=0}^{\infty}$ 
in the usual way by a descending intersection of the closures 
of tails of an orbit
\[
\omega((x_{n})) := \bigcap_{m=0}^{\infty} 
\overline{\{x_{n}: n{\geq} m \}}.
\]
One should be aware that in the case of IFSs and multivalued 
dynamical systems various kinds of omega-limit sets can be defined,
consult e.g. \cite{McGehee, Akin, Kieninger, Potzsche}. 
Comparing the nonautonomous discrete dynamical systems 
(\cite{nonautonomouSystems}) with the framework of IFSs 
(and also multivalued systems, cf. \cite{Igudesman, LasotaMyjak}), 
one sees that the orbit of the nonautonomous system is determined
by the starting point though the dynamics changes over time
and to determine the orbit of the IFS one needs additionally to
specify the driving sequence; loosely speaking in the theory
of IFSs we deal with the infinite number of nonautonomous  
systems upon a finite (sometimes countable 
\cite{MauldinUrbanskiGraph} or compact 
\cite{Wicks, Kieninger, LesniakCEJM}) set of generating maps.
Yet one can cast the IFS as a skew-product system,
e.g., \cite[Example 4.3]{Potzsche}. 

We finish this Section by giving motivating examples 
where the projection method invites a 
nonexpansive IFSs viewpoint (cf. \cite{Angelos}).
Specifically Example \ref{ex:BarnsleyTriangle}
suggests the way we should interpret the result of the
projection algorithm in general. We use a common notation
$H_{i}\subset X := {\mathbb{R}}^{2}$ for lines 
(hyperplanes) and $P_{i}:X \to H_{i}$ for orthogonal 
projections onto $H_{i}$ constituting a nonexpansive 
IFS $(X; f_{1},f_{2},{\ldots})$, $f_{i} := P_{i}$.

\begin{example}\label{ex:KaczmarzLines}
Given two lines intersecting at $x_{*}$ one projects alternately 
onto them to recover solution $x_{*}$ of the linear system.
The picture in Figure \ref{fig:KaczmarzLines} is the hallmark
of the projection method. We have the orbit
$P_{2}{\circ}P_{1}{\circ}P_{2}{\circ}P_{1}{\circ}P_{2}(x_0)$
converging to $x_{*}$. Note that the composition 
$P_{1}{\circ}P_{2}$ is contractive.

\begin{figure}
\caption{Alternating projections onto two lines.}
\label{fig:KaczmarzLines}
\begin{tikzpicture}[line cap=round,line join=round,>=triangle 45,x=1.0cm,y=1.0cm]
\clip(-0.24,2.61) rectangle (4.76,6.99);
\draw [domain=-0.24:4.76] plot(\x,{(--0.22-0.49*\x)/-0.26});
\draw [domain=-0.24:4.76] plot(\x,{(--2.28-0.51*\x)/0.31});
\draw [line width=1.2pt,dash pattern=on 2pt off 2pt] (0.29,5.68)-- (0.85,6.02);
\draw [line width=1.2pt,dash pattern=on 2pt off 2pt] (0.85,6.02)-- (3.03,4.87);
\draw [line width=1.2pt,dash pattern=on 2pt off 2pt] (3.03,4.87)-- (1.93,4.21);
\draw [line width=1.2pt,dash pattern=on 2pt off 2pt] (1.93,4.21)-- (2.52,3.9);
\draw [line width=1.2pt,dash pattern=on 2pt off 2pt] (2.22,3.73)-- (2.52,3.9);
\begin{scriptsize}
\fill [color=black] (2.33,3.54) circle (2.0pt);
\draw[color=black] (2.46,3.02) node {$x_{*}$};
\draw[color=black] (3.65,6.58) node {$H_1$};
\draw[color=black] (0.88,6.63) node {$H_2$};
\fill [color=black] (0.29,5.68) circle (1.5pt);
\draw[color=black] (0.16,5.42) node {$x_0$};
\fill [color=black] (0.85,6.02) circle (1.5pt);
\draw[color=black] (1.3,6.02) node {$x_1$};
\fill [color=black] (3.03,4.87) circle (1.5pt);
\draw[color=black] (3.38,4.77) node {$x_2$};
\fill [color=black] (1.93,4.21) circle (1.5pt);
\draw[color=black] (1.57,4.22) node {$x_3$};
\fill [color=black] (2.52,3.9) circle (1.5pt);
\draw[color=black] (2.8,3.83) node {$x_4$};
\fill [color=black] (2.22,3.73) circle (1.5pt);
\draw[color=black] (1.8,3.6) node {$x_5$};
\end{scriptsize}
\end{tikzpicture}
\end{figure}
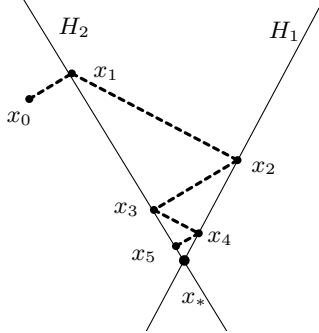
\end{example}

\begin{example}\label{ex:ParaLines}
Given two parallel lines one projects alternately 
onto them to get a pair of minimally distanced points.
The visualization provides Figure \ref{fig:ParaLines}.
Here $P_{1}{\circ} P_{2}$ is not contractive,
yet it behaves contractively on the orbit, 
see the last Section for precise formulation of this 
phenomenon.
Usually the case of parallel lines exhibits instability upon 
parameters for the solution problem of linear systems.

\begin{figure}
\caption{Alternating projections onto two parallel lines.}
\label{fig:ParaLines}
\begin{tikzpicture}[line cap=round,line join=round,>=triangle 45,x=1.0cm,y=1.0cm]
\clip(1.34,4.48) rectangle (4.44,7.19);
\draw [domain=1.34:4.44] plot(\x,{(--0.22-0.49*\x)/-0.26});
\draw [domain=1.34:4.44] plot(\x,{(-0.67-0.49*\x)/-0.26});
\draw [line width=1.2pt,dash pattern=on 2pt off 2pt] (1.97,6.3)-- (2.86,5.83);
\draw [line width=1.2pt,dash pattern=on 2pt off 2pt] (2.86,5.83)-- (3.39,5.55);
\begin{scriptsize}
\draw[color=black] (3.56,6.83) node {$H_1$};
\fill [color=black] (2.86,5.83) circle (1.5pt);
\draw[color=black] (2.56,5.56) node {$x_0$};
\draw[color=black] (2.56,6.74) node {$H_2$};
\fill [color=black] (1.97,6.3) circle (2.0pt);
\draw[color=black] (2.91,6.28) node {$x_2 = x_4$};
\fill [color=black] (3.39,5.55) circle (2.0pt);
\draw[color=black] (3.95,5.14) node {$x_1=x_3$};
\end{scriptsize}
\end{tikzpicture}
\end{figure}
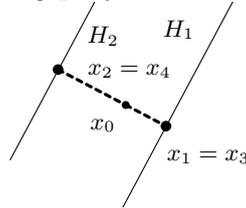
\end{example}

\begin{example}\label{ex:Square}
Suppose we have the following configuration: 
four lines so that each one is orthogonal to two others 
and parallel to the third one. Points of intersection,
denoted $y_{12}, y_{13}, y_{23}, y_{24}$, span a rectangle,
see Figure \ref{fig:Square}.
Then projecting onto these lines in a sufficiently ``random" manner,
e.g., $P_{2}{\circ}P_{3}{\circ}P_{1}{\circ}P_{4}
{\circ}P_{2}{\circ}P_{1}{\circ}P_{2}{\circ}P_{1}(x_0) $,
one quickly recovers the four corner points 
$C:=\{y_{12}, y_{13}, y_{23}, y_{24}\}$.  
Note that analogously to Example \ref{ex:KaczmarzLines} 
the composition $P_{3}{\circ} P_{1}$ 
($H_1$ and $H_3$ orthogonal) is contractive;
hence $P_{4}{\circ} P_{2}{\circ} P_{3}{\circ} P_{1}$
is contractive.
The minimal closed set invariant on the joint action
of all projections $P_{1},P_{2},P_{3},P_{4}$ 
is exactly the set $C$; see the next Section for the precise 
definition of an invariant set.

\begin{figure}
\caption{Projections onto four pairwise orthogonal or parallel lines.}
\label{fig:Square}
\begin{tikzpicture}[line cap=round,line join=round,>=triangle 45,x=1.0cm,y=1.0cm]
\clip(-0.12,3.37) rectangle (6.01,8.45);
\draw (4.25,3.37) -- (4.25,8.45);
\draw (2.01,3.37) -- (2.01,8.45);
\draw [line width=1.2pt,dash pattern=on 2pt off 2pt] (2.01,5.52)-- (1.26,5.52);
\draw [line width=1.2pt,dash pattern=on 2pt off 2pt] (1.26,5.52)-- (4.25,5.52);
\draw [domain=-0.12:6.01] plot(\x,{(--14.92-0*\x)/2.24});
\draw [domain=-0.12:6.01] plot(\x,{(--9.89-0*\x)/2.24});
\begin{scriptsize}
\draw[color=black] (3.95,7.6) node {$H_1$};
\fill [color=black] (1.26,5.52) circle (1.5pt);
\draw[color=black] (1.26,5.19) node {$x_0$};
\draw[color=black] (2.3,7.59) node {$H_2$};
\fill [color=black] (2.01,5.52) circle (1.5pt);
\draw[color=black] (2.75,5.13) node {$x_2 = x_4$};
\fill [color=black] (4.25,5.52) circle (1.5pt);
\draw[color=black] (5.0,5.41) node {$x_1=x_3$};
\fill [color=black] (4.25,6.66) circle (2.5pt);
\draw[color=black] (5.0,6.34) node {$y_{13}=x_7$};
\fill [color=black] (2.01,6.66) circle (2.5pt);
\draw[color=black] (2.75,6.25) node {$y_{23} = x_8$};
\fill [color=black] (2.01,4.41) circle (2.5pt);
\draw[color=black] (2.75,4.04) node {$y_{24} = x_5$};
\fill [color=black] (4.25,4.41) circle (2.5pt);
\draw[color=black] (5.1,4.03) node {$y_{12}=x_6$};
\draw[color=black] (0.86,6.99) node {$H_3$};
\draw[color=black] (0.72,4) node {$H_4$};
\end{scriptsize}
\end{tikzpicture}
\end{figure}
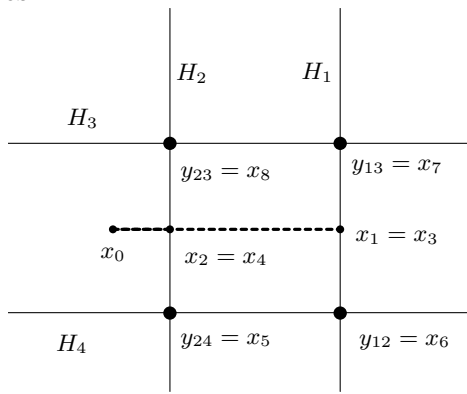
\end{example}

Since the orthogonal projection onto a hyperplane is 
nonexpansive w.r.t. the taxi-cab ${\ell}^1$-norm 
in the Euclidean space one might hope (in accordance 
with the Examples so far) that omega-limit sets
of systems consisting of orthogonal projections are finite
sets (\cite{AkcogluKrengel, Nussbaum, DiLena}).
This is not true as shown below.

\begin{example}\label{ex:BarnsleyTriangle}
If we project onto three lines each two of which are intersecting
and the choice of projections follows a disjunctive driver,
then the omega-limit set of such iteration 
constitutes a triangle with vertices at intersection points.
This phenomenon was observed quite long time ago
for drivers generated via discrete stochastic processes 
(`chaos game algorithm' \cite{BarnsleyPrivate}). 

We sketch the orbit 
$P_{1}{\circ}P_{3}{\circ}P_{1}{\circ}P_{2}
{\circ}P_{3}{\circ}P_{1}{\circ}P_{2}(x_0) $
in Figure~\ref{fig:BarnsleyTriangle}. 
Note that any two projections $P_{i}{\circ} P_{j}$, 
$i{\neq}j$, compose to contractions; in particular
$P_{3}{\circ}P_{2}{\circ}P_{1}$ is a contraction.

\begin{figure}
\caption{Randomly applied projections onto three lines.}
\label{fig:BarnsleyTriangle}
\begin{tikzpicture}[line cap=round,line join=round,>=triangle 45,x=1cm,y=1cm]
\clip(0.37,1.14) rectangle (4.79,5.01);
\draw [line width=3.2pt] (2.3,3.54)-- (4.12,1.57);
\draw [line width=3.2pt] (4.12,1.57)-- (1.41,1.8);
\draw [line width=3.2pt] (1.41,1.8)-- (2.3,3.54);
\draw [domain=0.37:4.79] plot(\x,{(--0.85-1.75*\x)/-0.9});
\draw [domain=0.37:4.79] plot(\x,{(--10.98-1.98*\x)/1.81});
\draw [domain=0.37:4.79] plot(\x,{(--5.19-0.23*\x)/2.71});
\draw [line width=1.2pt,dash pattern=on 2pt off 2pt] (2.25,4.57)-- (1.77,4.13);
\draw [line width=1.2pt,dash pattern=on 2pt off 2pt] (1.77,4.13)-- (2.43,3.79);
\draw [line width=1.2pt,dash pattern=on 2pt off 2pt] (2.43,3.79)-- (2.25,1.73);
\draw [line width=1.2pt,dash pattern=on 2pt off 2pt] (2.25,1.73)-- (3.19,2.58);
\draw [line width=1.2pt,dash pattern=on 2pt off 2pt] (3.19,2.58)-- (2.1,3.14);
\draw [line width=1.2pt,dash pattern=on 2pt off 2pt] (2.1,3.14)-- (1.98,1.75);
\draw [line width=1.2pt,dash pattern=on 2pt off 2pt] (1.98,1.75)-- (1.51,1.99);
\begin{scriptsize}
\draw[color=black] (3.1,4.44) node {$H_1$};
\draw[color=black] (1.1,4.37) node {$H_2$};
\draw[color=black] (0.9,2.1) node {$H_3$};
\fill [color=black] (2.25,4.57) circle (1.5pt);
\draw[color=black] (2.4,4.79) node {$x_0$};
\fill [color=black] (1.77,4.13) circle (1.5pt);
\draw[color=black] (1.56,3.99) node {$x_1$};
\fill [color=black] (2.43,3.79) circle (1.5pt);
\draw[color=black] (2.86,3.7) node {$x_2$};
\fill [color=black] (2.25,1.73) circle (1.5pt);
\draw[color=black] (2.3,1.48) node {$x_3$};
\fill [color=black] (3.19,2.58) circle (1.5pt);
\draw[color=black] (3.49,2.64) node {$x_4$};
\fill [color=black] (2.1,3.14) circle (1.5pt);
\draw[color=black] (1.78,3.21) node {$x_5$};
\fill [color=black] (1.98,1.75) circle (1.5pt);
\draw[color=black] (1.95,1.5) node {$x_6$};
\fill [color=black] (1.51,1.99) circle (1.5pt);
\draw[color=black] (1.34,2.16) node {$x_7$};
\end{scriptsize}
\end{tikzpicture}
\end{figure}
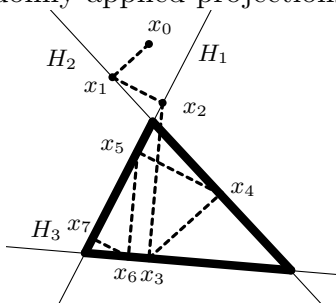
\end{example}
 
The reader is encouraged to dwelve in \cite{Angelos} for more examples 
with detailed analyses of polygonal omega-limit sets.

\section{Generalities}

We shall present here a general relationship between invariant sets and 
omega-limit sets. Throughout let $C\subset X$ denote a nonempty
closed bounded subinvariant set of the nonexpansive IFS
$(X;f_i,i=1,{\ldots},N)$, $\Phi$ the Hutchinson operator and 
$(x_{n})_{n=0}^{\infty}$ the orbit.

Given a nonempty set $S\subset X$, we employ also 
the notation:
\begin{itemize}
\item $d(p,S) := \inf_{s\in S} d(p,s)$ for the distance from the point 
$p\in X$ to $S$;
\item $N_{\varepsilon}S := \{x\in X: d(x,S) < \varepsilon\}$ for 
the $\varepsilon$-neighbourhood of $S$.
\end{itemize}

It turns out to be convenient to use a known 
sequential characterization of the omega-limit set:
\begin{equation*}
x_{*}\in\omega((x_{n})) \;\text{iff}\; x_{k_n}\to x_{*}
\;\text{for some subsequence}\; k_{n}\nearrow\infty.
\end{equation*}
Whenever we make statements about omega-limit sets we 
\emph{assume} that they are nonempty; for empty
omega-limit sets the statements are void. 
In this respect we have the following basic criterion.
\begin{lemma}[\cite{McGehee, Akin, nonautonomouSystems}]
If the orbit $(x_{n})_{n=0}^{\infty}$ is precompact, then
$\omega((x_{n}))$ is nonempty and compact. 
\end{lemma}
Recall that the orbit $\{x_{n}\}_{n=0}^{\infty}$ in a complete 
space $X$ is precompact if it has compact closure.

The simple observations below are crucial for 
basic relationship between invariant sets and omega-limit sets. 
\begin{lemma}\label{monotonedist}
Let $(x_{n})_{n=0}^{\infty}$ be the orbit and suppose that there
exists a nonempty closed bounded set $C$ which is subinvariant,
$\Phi(C)\subset C$. Then
\begin{enumerate}
\item[(i)] $d(x_{n+1},C)\leq d(x_{n},C)\leq d(x_{0},C)$,
\item[(ii)] $(x_{n})_{n=0}^{\infty}$ is bounded,
\item[(iii)] $d(y,C) = \inf_{n} d(x_{n},C) \equiv \text{const}$ 
for $y\in\omega((x_{n}))$.
\end{enumerate}
\end{lemma}
\begin{proof}
Fix ${\varepsilon} >0$. Find appropriate $c_0\in C$ to write
\begin{gather*}
d(x_0,C)+\varepsilon \geq d(x_0,c_0) \geq
d(f_{i_1}(x_0),f_{i_1}(c_0)) = \\
= d(x_{1},f_{i_1}(c_0)) \geq d(x_{1},C); 
\end{gather*}
the last inequality relies on $f_{i_1}(c_0)\in \Phi(C)\subset C$.
By induction (i) follows. 

Moreover $x_{n}\in N_{2d(x_{0},C)}C$ for $n\geq 1$, where 
the latter set is bounded. Hence (ii) follows.

The sequence $d(x_{n},C)$ is monotonely decreasing and 
bounded from below by $0$, thus it is convergent to
$\text{const} = \inf_{n} d(x_{n},C)$. Let $x_{k_n}\to y$.
By monotonicity again
\[
d(x_{k_n},C) \to \inf_{n} d(x_{k_n},C) = \text{const},
\] 
but continuity of the distance yields $d(x_{k_n},C)\to d(y,C)$.
Therefore we have (iii).
\end{proof}

\begin{proposition}\label{intersectomega}
If a closed bounded subinvariant set $C$ intersects the omega-limit
set, $C\cap \omega((x_{n}))\neq\emptyset$, then it contains that 
omega-limit set, $C\supset\omega((x_{n}))$. In particular,
the omega-limit set is the minimal invariant set, 
provided it is invariant.
\end{proposition}
\begin{proof}
Let $y_0\in C\cap \omega((x_{n}))$. 
From Lemma \ref{monotonedist} we know that any
$y\in \omega((x_{n}))$ necessarily obeys $d(y,C)=d(y_0,C)=0$.
\end{proof}

Note that the set $C$ in the above proposition need not be bounded
(as careful analysis of the proof of Lemma \ref{monotonedist} (i), (iii)
shows). Although the omega-limit set need not be invariant 
there holds
\begin{proposition}[\cite{BarnsleyPrivate}]\label{superinvomega}
The omega-limit set is superinvariant, i.e.,
\[
\Phi(\omega((x_{n}))\,)\supset \omega((x_{n})).
\]
\end{proposition}
\begin{proof}
Let $x_{*}\in \omega((x_{n}))$, $x_{k_n}\to x_{*}$.
Then there exists a symbol $\sigma\in\{1,{\ldots},N\}$
s.t. $x_{{k_n}+1}=f_{\sigma}(x_{k_n})$ for infinitely 
many $k_n$'s, say $k_{l_n}$. Hence
\[
x_{k_{l_n}+1}= f_{\sigma}(x_{k_{l_n}}) \to 
f_{\sigma}(x_{*}) \in\Phi(\omega((x_{n}))\,).
\]
\end{proof}

For good properties of orbits and omega-limit sets one may 
need the notion of a proper space. 
A metric space $(X,d)$ is called \textbf{proper} 
(cf. \cite{BarnsleyVinceChaos}) if every ball not equal
to the whole space has compact closure; 
the terminology is not standardized, e.g., Beer uses
the term `space with nice closed balls', 
see \cite{Beer} 5.1.8 p.142.
Necessarily any such space is complete and locally compact. 
Conversely, a locally compact space is proper after suitable 
remetrization (\cite{Beer} 5.1.12 p.143). 
The main advantage of $X$ being proper is that any 
bounded orbit is precompact and consequently admits
nonempty compact omega-limit set.
We would like to stress out that our further considerations 
do not rely on the properness.

\section{Main theorem}

Let $\Phi = (X;f_1,{\ldots},f_N)$ be a nonexpansive IFS 
on a complete space $X$. We constantly assume that 
$\Phi$ possesses a nonempty closed bounded subinvariant 
set. Hence by observations made in the previous Section
all orbits and omega-limit sets are warranted to be bounded.
This is not very restrictive hypothesis, because if the 
omega-limit set recovers an invariant set, then there must
be present at least one (sub)invariant set.

Suppose $(x_{n})_{n=0}^{\infty}$,
$(y_{n})_{n=0}^{\infty}$ are two orbits generated
by the driving sequences $(i_{n})_{n=1}^{\infty}$,
$(j_{n})_{n=1}^{\infty}$
and starting at $x_{0},y_{0}\in X$, respectively.

\begin{proposition}\label{disjunctiveomega1}
Let us assume that
\begin{enumerate}
\item[(D)] the driving sequences for orbits
$(x_{n})_{n=0}^{\infty}$ and $(y_{n})_{n=0}^{\infty}$
are disjunctive,
\item[(C)] there exists a sequence of symbols 
$(u_1,{\ldots},u_{l})\in\{1,{\ldots},N\}^{l}$ s.t.
the composition $f_{u_{l}}\circ{\ldots}\circ f_{u_{1}}$
is a Lipschitz contraction.  
\end{enumerate}
Then the omega-limit set does not depend on 
the initial point
\[\omega((x_n)) = \omega((y_n)).\]
\end{proposition}
\begin{proof}
We additionally assume that 
the driving sequences are the same for both
orbits, $i_{n}=j_{n}$. The general case 
will follow from Lemma~\ref{disjunctiveomega}
below.

Denote by $L<1$ the Lipschitz constant of 
$f_{u_{l}}\circ{\ldots}\circ f_{u_{1}}$. 
By disjunctivity of the driver $(i_{n})_{n=1}^{\infty}$,
given any subsequence $k_{n}$ there exists a deeper
subsequence $k_{l_n}$ s.t. 
$(i_{k_{l_n}}, i_ {k_{l_n}-1}, {\ldots}, i_{1})$ contains
as a subword the sequence 
$(u_{l},{\ldots},u_{1},{\ldots}, u_{l},{\ldots},u_{1})
\in \{1,{\ldots},N\}^{m_{n}{\cdot} l}$, i.e., 
$(u_{l},{\ldots},u_{1})$ repeated $m_{n}$-times,
and additionally $m_{n}\nearrow\infty$.
Therefore
\[
d(x_{k_{l_n}},y_{k_{l_n}}) \leq L^{m_n}\cdot 
d(x_0,y_0)\to 0.
\]
This means that both orbits admit the same limit 
points.
\end{proof}

We are going to strengthen this result, 
by showing under contractivity condition substantially 
weaker than (C), that the omega-limit sets of orbits 
starting at the same point are identical provided 
the sequences driving these orbits are disjunctive.

\begin{lemma}\label{disjunctiveomega}
Let us assume about orbits 
$(x_{n})_{n=0}^{\infty}$,
$(y_{n})_{n=0}^{\infty}$
that they start from the same point $y_0=x_0$,
and obey (D) and
\begin{enumerate}
\item[(CO)] there exists a sequence of symbols 
$(u_1,{\ldots},u_{l})\in\{1,{\ldots},N\}^{l}$ s.t.
the composition $f_{u_{l}}\circ{\ldots}\circ f_{u_{1}}$
is a Lipschitz contraction when restricted to the set
\begin{gather*}
\bigcup_{n=0}^{\infty} \Phi^{n}(\{x_0\}) = 
\{x_0, f_1(x_0), {\ldots}, f_{N}(x_0),  \\
f_{1}\circ f_{1}(x_0), f_{1}\circ f_{2}(x_0), {\ldots}, 
f_{1}\circ f_{N}(x_0), {\ldots},  \\
 f_{N}\circ f_{1}(x_0), {\ldots}, f_{N}\circ f_{N}(x_0), {\ldots}\}
\end{gather*}
(a branching tree with root at $x_0$). 
\end{enumerate}
Then the omega-limit sets coincide
\[\omega((x_n)) = \omega((y_n)).\]
\end{lemma}
\begin{proof}
Denote by $L$ the Lipschitz constant of 
$f_{u_{l}}\circ{\ldots}\circ f_{u_{1}}$.
Fix $x_{*}\in \omega((x_{n}))$, $x_{k_{n}}\to x_{*}$.
Similarly as in the proof of Proposition \ref{disjunctiveomega1}
there exists $k_{l_n}$ s.t. 
$(i_{k_{l_n}}, {\ldots}, i_{1})$ contains
$(u_{l},{\ldots},u_{1})$ repeated 
consecutively $m_{n}$-times, $m_{n}\nearrow\infty$.
Then, due to disjunctivity of $j_{n}$, we can find
a subsequence $j_{r_n}$ s.t. 
\begin{gather*}
f_{j_{r_n}}= f_{i_{k_{l_n}}}, \\
f_{j_{(r_{n}-1)}}= f_{i_{(k_{l_n}-1)}},\\
{\ldots},
f_{j_{(r_{n}-(k_{l_n}-1))}} = f_{i_1}.\\
\end{gather*}
Hence we arrive at
\[
d(x_{k_{l_n}},y_{r_n})\leq L^{m_n}\cdot 
d(x_0,y_{{r_n}-{k_{l_n}}}) \to 0.
\]
Therefore $x_{*}\in\omega((y_n))$.
\end{proof}

Now we establish the main result of the whole article.

\begin{theorem}\label{OmegaThm}
If the orbit $(x_{n})_{n=0}^{\infty}$ of the 
nonexpansive IFS $(X;f_1,{\ldots},f_N)$ is bounded
and driven by a disjunctive sequence of symbols and
the system has the property (CO) of 
contractivity on orbits, then $\omega((x_{n}))$ 
is a minimal closed invariant set. Under stronger condition
(C) of contractivity, any two omega-limit sets generated
by a disjunctive choice of maps coincide.
\end{theorem}
\begin{proof}
By Proposition \ref{intersectomega} (or by Proposition
\ref{superinvomega} if one prefers) we only need to 
prove subinvariance:
\[
f_{\sigma}(\omega((x_{n}))\,) \subset \omega((x_{n}))
\;\text{ for every }\; \sigma\in\{1,{\ldots},N\}.
\]
Let $f_{\sigma}(x_{*})\in f_{\sigma}(\omega((x_{n}))\,)$,
$x_{k_n}\to x_{*}$. Take a finer subsequence 
$k_{l_n}$  according to the following rules: the sequence
$(i_{k_{(l_{n-1})}+2},{\ldots},i_{k_{l_n}})$ contains 
\begin{enumerate}
\item[(a)] all finite words of length $n$ 
build over the alphabet $\{1,{\ldots},N\}$,
\item[(b)] $m_{n}$-times repeated 
word $(u_1,{\ldots},u_{l})$
appearing in the condition (CO), $m_{n}\nearrow\infty$. 
\end{enumerate}
Redefine $i_{n}$ in such a way that 
$\tilde{i}_{k_{l_n}+1} = \sigma$ and 
$\tilde{i}_{n}= i_{n}$ otherwise. The sequence 
$\tilde{i}_{n}$ is disjunctive due to (a). Hence
the orbit $y_{n}$ starting at $y_0:= x_0$ with
the driver $j_{n} := \tilde{i}_{n}$ has the property
that $d(y_{k_{l_n}},x_{k_{l_n}}) \to 0$ as warranted
by (b). So 
\[
y_{k_{l_n}+1} = f_{\sigma}(y_{k_{l_n}}) \to 
f_{\sigma}(x_{*})\in \omega((y_{n})).
\]
Therefore 
\[
\omega((x_{n})) = \omega((y_{n})) \ni f_{\sigma}(x_{*})
\]
via Lemma \ref{disjunctiveomega}.
\end{proof}

\begin{example}
Consider a system $(X; f_{i}, i=1,{\ldots},N)$ comprising 
of orthogonal projections $f_{i}:= P_{i}$ onto affine 
subspaces $H_{i}\subset X$ in the Euclidean space $X$.
Due to \cite{Meshulam} we know that any orbit produced
by projections is bounded. 
Therefore in the standard situation we do not need to 
assume that there exists a bounded subinvariant set.
Moreover Theorem \ref{OmegaThm} warrants the 
existence of a bounded invariant set.
 
The most important fact about orthogonal 
projections is that the composition 
$P_{N}{\circ}{\ldots}{\circ}P_{1}$ 
is contractive on orbits, obeys condition (CO). 
Several results in this direction have been
obtained throughout the years: 
\cite{Kosmol, BauschkeAs, Kirchheim}.

This settles the case of the Kaczmarz algorithm 
when projecting onto multiple hyperplanes with 
empty intersection (cf. \cite{Angelos}).
\end{example}

\end{document}